\documentclass[article,noinfoline]{imsart}
\usepackage{graphicx,enumerate}

\usepackage{jr2,comment}

\newtheorem{theorem}{Theorem}
\newtheorem{lemma}{Lemma}
\newtheorem{Assumption}{Assumption}
\newtheorem{corollary}{Corollary}
\newtheorem{example}{Example}
\startlocaldefs
\def\bb#1{\mathbb{#1}}

\def\boldsymbol#1{\setbox\ewb\hbox{$#1$}%
    \setlength{\deno}{-\wd\ewb+0.05em}{ #1}\hspace{\deno}{#1}}
\endlocaldefs

\begin{document}

\baselineskip=22pt

\title{Consistent Empirical Bayes estimation  of the mean of a mixing distribution without identifiability assumption. With applications to
treatment of non-response.}

\begin{center}
{ {\bf \large  Eitan Greenshtein}

Central Bureau of Statistics, Israel

eitan.greenshtein@gmail.com
} 
\end{center}
{\bf Abstract}

Consider a Non-Parametric  Empirical Bayes  (NPEB) setup. We observe $Y_i, \sim f(y|\theta_i)$, $\theta_i \in \Theta$ independent, where $\theta_i \sim G$
are independent $i=1,...,n$. The mixing distribution  $G$ is unknown $G \in \{G\}$   with no parametric assumptions about the class $\{G \}$. The common NPEB task is to estimate 
$\theta_i, \; i=1,...,n$. Conditions that imply 'optimality' of such NPEB estimators typically require identifiability of  $G$ based on
$Y_1,...,Y_n$. 
We consider the task of estimating $E_G \theta$. We show that `often' consistent estimation of $E_G \theta$ is implied without
identifiability.

We motivate the later task, especially in setups with non-response and missing data. We demonstrate consistency in simulations.

\section{Introduction}

Our goal in this paper is to utilize Non Parametric  Empirical Bayes (NPEB) ideas for the purpose of estimation of means of mixtures, in particular under sampling with non-response.

Let $\theta_i \sim G$, $i=1,...,n$ be iid, let  $Y_i \sim f(y|\theta_i), \; \theta_i \in \Theta$;  $f(y|\theta) $  is the density of the distribution of $Y$, with respect to  some dominating  measure $\mu$ under $\theta$. The conventional and most common  goal in
NPEB is point estimation, i.e., the estimation of $\theta_i$ based on the observed $Y_1,...,Y_n$, without knowing $G$.

Given a Generalized  Maximum Likelihood  Estimator  (GMLE) $\hat{G}$  for $G$, a plausible and common estimator for $\theta_i$ is
$E_{\hat{G}} (\theta|Y_i)$.

The above is trivially generalized when the goal is to estimate 
$\eta(\theta_i) \equiv \eta_i$ for some function $\eta$. Then we have
the point estimators 
$$\hat{\eta}_i=E_{\hat{G}} (\eta(\theta)|Y_i)$$

We will consider the less  common practice of estimating $E_G \theta$,
or more generally, $E_G \eta(\theta) \equiv \eta_G$, especially for the case $\eta(\theta)=E_\theta(h(Y))$ for a given $h$.  Thus, $E_G(\eta(\theta))= E_G h(Y)$. See Greenshtein and Ritov  (2022) (GR),
see also Greenshtein and Itskov (2014) (GI). A plausible estimator is
$$\hat{\eta}_G \equiv E_{\hat{G}} \eta(\theta)=\eta_{\hat{G}}.$$
Note the above estimator is not well defined when $\hat{G}$ is not unique, handling this issue is  one of the main  goals of this paper. In (GI) and (GR)  the above estimator was understood as the set of all of its
values that correspond to all possible GMLE  $\hat{G}$. 

The above estimator is useful  when one has a weak convergence of {\it any} sequence of GMLE
$\hat{G}$ to $G$, as $n \rightarrow \infty$.    This is since that for any continuous and bounded $\eta(\theta)$, $\hat{G} \Rightarrow_w G$,
implies  $ \hat{\eta}_G \rightarrow \eta_G$.
On weak convergence of  a GMLE, see Kiefer and Wolfowitz (1956), see also Chen, J.  (2017).
 An important condition that is required
in order to have such a weak convergence is {\it identifiabiliity}.
We were surprized in (GR) and (GI) to see in simulations that
 $\hat{\eta}_G$ estimates $\eta_G$ excelently even when there is no
identifiability, and any possible GMLE $\hat{G}$ was applied.

Understanding the last phenomena is the main motivation for this paper.
In the next subsection, we  will elaborate on the conventional
NPEB problem of  point estimation, versus the NPEB estimation of the mean of a mixture.
The later task is important in official statistics, where the main goal is to estimate population's mean
rather than individual parameters.

First some basic definitions.

\noindent{\bf GMLE for $\boldsymbol{G}$.}
Given a distribution $G$ and a dominated family of distributions with densities {\nolinebreak$\{ f(y\mid \theta):\; \theta \in \Omega \}$} with respect to some dominating measure $\mu$, define
$$ f_G(y)= \int f(y\mid \vartheta) dG(\vartheta).$$

Given observations $Y_1,...,Y_n$,
a {\it GMLE} $\hat{G}$ for $G$ (Kiefer and Wolfowitz, 1956) is defined  as (any) $\hat{G}$ satisfying:
\begin{equation} \label{eq:hatG} \hat{G}= \argmax_{{G}} \; \Pi f_{{G}}(Y_i); \end{equation}
the maximization is with respect to all probability distributions $\{G\}$ on $\Theta$.

Under the above setup we say that we have {\it identifiablity} if for every
$G_1$  and $G_2$, $f_{G_1}=f_{G_2}$ implies $G_1=G_2$.

\subsection {NPEB for point estimation versus NPEB for the estimation of the mean of a mixture}\label{sec:versus}

The natural estimator for $\eta_G=E_G(h(Y))$ in the above  setup
is  $$\frac{1}{n} \sum h(Y_i).$$ 
The last estimator is also efficient, e.g., under an exponential family setup. It can not be significantly improved  in general.  However the problem becomes interesting and challenging in situations with non-response.
NPEB could be useful in the later situations, as will be demonstrated.

The following example demonstrates the relation and the difference between the two  estimation tasks. It is not yet in the context of  non-response.

\begin{example} \label{ex:first}
Let  $\theta_i \sim G$ be iid, $Y_i \sim Bernoulli(\theta_i)$, $\theta_i \in [0,1]$, $i=1,...,n$.  Suppose $Y_1=...=Y_{n/2}=0$ and $Y_{n/2+1}=...= Y_n=1$.

It is easy to see that  both the following $\hat{G}_1$ and  $\hat{G}_2$ are GMLE
for this data set. Let $\hat{G}_1$ satisfy $P_{\hat{G}_1}(\theta=0.5)=1$, while $\hat{G}_2$ satisfy $P_{\hat{G}_2}(\theta=0.75)=P_{\hat{G}_2}(\theta=0.25)=0.5.$

Note, the above two distributions are non-identiafiable.  Also,
$$E_{\hat{G}_1}(\theta_i|Y_i=1)\neq  E_{\hat{G}_2}(\theta_i|Y_i=1).$$
Thus, NPEB is useless  for the task of point estimation, because it is not clear which GMLE should be used.

However, for $h(Y)=Y$, and $\eta(\theta)= E_{\theta} h(Y)$ 
$$ E_{\hat{G}_1}\eta(\theta)=E_{\hat{G}_2}\eta(\theta)=0.5.$$
A more general result, under   exponential family setups  is shown in
(GR) 2022, where it is shown that for any GMLE $\hat{G}$,
$E_{\hat{G}} Y=\frac{1}{n}\sum_i Y_i$. 

Note, in the setup of the last example, in the case $\eta(\theta)=\theta^2$,
$ E_{\hat{G}_1}\eta(\theta) \neq E_{\hat{G}_2}\eta(\theta).$ A key difference between the two cases will be clarified  in the sequel.

So, the estimator $\hat{\eta}_G=E_{\hat{G}} \eta(\theta)$ might be useful  for the estimation of $\eta_G$ inspite of the non-identifiability. This is since  that "sometimes"
  we may apply {\it any} GMLE.
\end{example}
 
\bigskip

It is worth mentioning  that, in general, for any  GMLE $\hat{G}$,
$E_{\hat{G}}\eta(\theta)=\sum_i E_{\hat{G}} (\eta(\theta)| Y_i)$. See (GR) 2022.

\subsection{Empirical Bayes, a brief review.}
The idea of Empirical Bayes was suggested by Robbins, see  Robbins (1953 ), (1956), (1964).
In Parametric Empirical Bayes, the prior $G$ is assumed to belong to a parametric family of distribution and the task
is to estimate $G$, or its corresponding parameters, based on the observed $Y_1,...,Y_n$.
As an example consider the case  $f(y|\theta)=N(\theta,1)$, $\theta \in \Theta \subseteq \R$ and $G=N(0,\sigma^2)$, where $\sigma^2$ is unknown.  In Non-Parametric 
Empirical Bayes  the prior $G$ is  a member of the set of all possible distributions on the parameter set $\Theta$.

In the non-parametric case  there is the $f$-modelling approach under which one attempts to estimate the density 
$f_G(y)$ and $E_G(\theta|Y_i)$ indirectly without estimating $G$. Another approach is the $G$-modelling under which
$G$ is estimated, see Efron (2014) for the two approaches.  See reviews on Empirical Bayes and its applications, e.g., in Efron ( 2010), and in  
Zhang ( 2003).

The estimation of $G$, under non-parametric  $G$-modelling, is typically done through GMLE. Traditionally the estimation was performed using EM-algorithm, see Laird (1978). Koeneker and Mizera ( 2014) suggested modern convex  optimization techniques  for the purpose of computation.

As mentioned the more  common task in Empirical Bayes is the point estimation of the individual parametrs $\theta_i$,  where $Y_i \sim f(y|\theta_i)$, $i=1,...,n$. In this paper, we emphasize the task of estimating the mean of the mixture $G$.

\section{Examples.} \label{sec:ex}

We  introduce some motivating examples explored also in
(GR) and in (GI).

The following example is studied in (GI) (2014).

\begin{example} \label{ex:GI}
Consider a survey, under which each sampled item is approached $K$
times at most. If no response is obtained within $K$ attempts the outcome is " non-response". The  response is denoted $X$, $X=x_1,...,x_S$, e.g., $x_1=\mbox{"employed"}, \; x_2=\mbox{ "unemployed"}$, as in (GI). 
Let $K_i$  be the number of attempts until a response, $Y= (X_i, K_i)$, however we oobserve $Y_i$ that correspond to sampled item $Y_i$,  $K_i \leq K$, otherwise it is censored and we observe "non-response",
or equivalently we observe that the number of attempts until response was greater than $K$. The possible
values of an observed  $Y$ are $(X_i, K_i)$ and "Non-Response" altogether
$K \times S +1$ possible values for an observed $Y$ (including non-response).

Each sampled item $i$ has  a random   response which is determined by 
a parameter $\theta_i$, $i=1,...,n$. In the above example think of the observed number of attempts $K_i$ 
as a truncated Geometric random variable, with parameter $\pi^i$, the probability of $X_i=x_s$ under $\theta_i$ is $p_i^s, \; s=1,2,...,S$, where $\theta_i=(\pi^i, p^1_i,...,p_S^i)$. Assume that conditional on $\theta_i$, $K_i$ is independent of $X_i$.
One may further extend the model to allow $\pi^i=(\pi_1^i,...,\pi^i_K)$,
where $\pi^i_k$ is the probability under $\theta_i$, of item $i$ to respond at the $k'th$ attempt. Extending the model too much may result in overfitting.
It is desired to estimate $E_G \eta(\theta)$ where $\eta(\theta_i)=(p^i_1,...,p^i_S)$. This estimates the population's  proportion of items with $X=x_s$, e.g., proportion of unemployed in the population (assuming  the original sample including the non-rsponse items
is representing).

\end{example}
The following example is studied in (GR) 2022.

\begin{example} \label{ex:GR}
a) It is desired to estimate the proportion in the population of (say) unemployed. A sample  is designed
in "small areas/ strata",  in each of $n$ strata.  In order to have approximate "missing at random" conditional on the strata, strata are chosen"small".
A sample of size $\kappa$ is taken from each of the $n$ strata. Let $\kappa_i$ be the number of responses 
in starata $i$, $i=1,...,n$. The possible outcome are $Y_i=(X_i, \kappa_i)$, $i=1,...,n$, where $X_i$ is the number of unemployed among the $\kappa_i$ responders, while $\kappa_i=0,...,\kappa$ is the number of responders in strata $i$. Assume that the conditional
distribution of $X_i$  conditional  on $\kappa_i$ is $B(\kappa_i, p_i)$, while $\kappa_i \sim B(\kappa,  \pi_i)$ . 
Let $\theta_i=(\pi_i,p_i)$.
It is desired to estimate $E_G \eta(\theta)$  where $\eta(\theta_i)= p_i$, e.g., the population's proportion  $E_G \eta(\theta)$   (assuming strata are of equal size) .

b) A similar setup is suitable under observational studies, where $\kappa_i$,   the number of observed items in strata $i$, is distributed $Poisson(\lambda_i)$. This is, e.g., when it is desired to estimate
the spread of a disease. The estimator is  based on a "convenience sample" from small areas reporting   on the number of people in those areas,
that had an (unrelated) blood test and the results of the test.  
Here the parmeter of interest is $\theta_i=(\lambda_i,p_i)$, $\eta(\theta_i)=p_i$.

Note, in the  current case, where $\kappa_i$ is distributed  Poisson, (GR) proved  that $\hat{G} \Rightarrow_w G$, so good simulation and real data analysis, reporeted in (GR), are less surprising.

\end{example}

\bigskip
 \section{On the uniquness of $\hat{\eta}_G$, in spite of non-identifiability}

\subsection{Uniqness of the distribution of the observed $y_i$, under GMLE with non-identifiability} 

Given observed $Y_1=y_1,..., Y_n=y_n$, let $\hat{G}_1$ and $\hat{G}_2$ be two GMLE. Then:

\begin{lemma} \label{lem:main}
For each   $y_i$, $i=1,...,n$, $$f_{\hat{G}_1}(y_i)=f_{\hat{G}_2}(y_i).$$
\end{lemma}

Before proving the above we need the following preperations.

Recall,  $\hat{G}=\argmax_{G} \Pi_i  f_G(Y_i) \equiv \argmax_{G} \log( \Pi_i  f_G(Y_i) ) 
\equiv \argmax_G \log(L(G)).$

The functional $\log(L(G))$ is concave, but not necessarily strictly concave. Thus, given two maximizers, i.e., two GMLE $\hat{G}_1, \hat{G}_2$, for any 
$\lambda \in (0,1)$ also $\lambda\hat{G}_1 +(1-\lambda)\hat{G}_2 \equiv  G_\lambda$ is a GMLE.

Denote by $ g_i= \frac{ d\hat{G}_i}{d\nu}$ the density of $\hat{G}_i$, with respect to some dominating measure $\nu$,  $ i=1,2$. 
 We obtain that: for any $\lambda \in (0,1)$
 
\begin{eqnarray}
\frac{d}{d\lambda} \log(  L(G_\lambda)) &=& \frac{d}{d\lambda}\sum_i   \log [\int (  f(y_i|\vartheta) [\lambda g_1(\vartheta)+(1-\lambda) g_2(\vartheta) d\nu(\vartheta) ]\\
&=& \sum_i 
\frac{  f_{\hat{G}_1}(Y_i) +  f_{\hat{G}_2}(Y_i)  } 
{ \lambda f_{\hat{G}_1}(Y_i) + (1-\lambda)  f_{\hat{G}_2}(Y_i)  }=0.  
\end{eqnarray}

Denote $c_i=\frac{ f_{\hat{G}_1}(Y_i)  }{ f_{\hat{G}_1}(Y_i) }  $.

Suppose that the above derivative is zero at $\lambda_0 \in (0,1)$. It may be verified that, unless $c_i \equiv 1$, it can {\it not} be zero
at $\tilde{\lambda}=\lambda_0+ \epsilon$, for any $\epsilon >0$ small enough so  that $\tilde{\lambda} \in (0,1)$. This may be seen by splitting the above sum into two; one with $c_i \leq 1$ and another with $c_i>1$.

Now, the proof of the above lemma is immediate.
 \begin{proof}
The proof follows, since $c_i \equiv 1$ implies  $f_{\hat{G}_1}(Y_i)  = f_{\hat{G}_2} (Y_i), \; i=1,...,n.$
\end{proof}

\subsection{Uniqueness of $\eta_{\hat{G}}$  uder non-identifiability.}\label{sec:unique}

\begin{Assumption}\label{as: cond}
There exists a function $h$ and a set $A$ such that for every $\theta \in \Theta$
i)$$\eta(\theta)= E_\theta( (h(Y)|A).$$

ii) $\inf_\theta P_\theta(Y \in A)>0.$
\end{Assumption}

Part i) in the above assumpion is in the spirit of Missing At Random Conditional on $\theta$, see 
little and Rubin (2002). Unlike 
the setup in MAR, we will be able to make inference even if for many $\theta_i$ we have no observation $Y_i$, satisfying $Y_i \in A$.

In both of our motivarting examples, Examples \ref{ex:GI}, \ref{ex:GR}, the above assumption is satisfied.
In the first case take $A$ to be the of all $y=(x,k)$ such  that $k \leq K$, in the second case  the set of all 
$y=(x,\kappa^*)$ such that $\kappa^*>0$.  For $y \in A$ the corresponding $h$ are $h(y)=(I(x=x_1),...,I(x=x_S))$  in the first example and $h(y)=\frac{x}{\kappa^*}$, in the second.   In both cases, we need to assume 
that  there exists $\delta>0$ so that $P_\theta (Y \in A)>\delta$  for every 
$\theta \in \Theta$, and thus,  part ii) of the last assumpton is
satisfied.

 We  first treat the case where $A=\Omega$ is the entire sampled set.

By  definition and the last  assumption we obtain:

\begin{eqnarray}
\eta_G &=&\int \eta(\vartheta) dG(\vartheta) =\int \int h(y) f(y|\vartheta) d\mu dG(\vartheta)  
\label {eqn:new}\\
&=& \int \int h(y) f(y|\vartheta)  dG(\vartheta) d\mu\\
&=& \int  h(y) f_G (y) d\mu \equiv E_G h(Y). \label{eq:bottom}
\end{eqnarray}

The  above is true also when replacing $G$ by a GMLE $\hat{G}$. The main point is that the bottom line in
equation (\ref{eq:bottom}), depends on $G$ only through $f_G$, similarly for $\hat{G}$ and $f_{\hat{G}}$.
Also $f_{\hat{G}}(y)$ is fixed  for any GMLE $\hat{G}$,  when $y=Y_i$ for some observation $Y_i$. 

In the above we applied Fubini and exchanging order of integration. We add the appropriate assumption.

\begin{Assumption} \label{as: fubini} 
We assume throughout,  the appropriate conditions that imply  Fubini's theorm and allowing
exchange in the order of integration.
\end{Assumption}

 We now,   consider the above for the case $A \subset \Omega$.

Consider an empirical Bayes setup with truncated  observations ( i.e, we do not know about obseraions $Y$, $Y \notin A$)  and corresponding densities
$$ \{ f^A (y|\theta) \equiv  \frac{f(y|\theta)}{P_\theta (A) }\times I(Y\in A), \; \theta \in \Theta \}.$$

As before, let $$f_{\hat{G}}^A(y)= \int f^A(y|\vartheta) d\hat{G}(\vartheta).$$

Then by the Assumption \ref{as: cond}:
\begin{eqnarray}
\eta_{\hat{G}} =E_{\hat{G}} \eta(\theta)&= &\int \int  h(y)f^A(y|\vartheta) d\mu d\hat{G}(\vartheta) \\
&=&\int h(y) f_{\hat{G}}^A (y) d\mu
\end{eqnarray}
 
Now, in the above, by Lemma \ref{lem:main},  for $y$ satisfying $y=Y_i$ for some $Y_i$,
 $ f_{\hat{G}^1}(y)=f_{\hat{G}^2}(y)$,
when $\hat{G}^i \; i=1,2$ are GMLE based on the truncated sample. In order to have it for GMLE
that is based  in addition on  the censored information in the truncated  observations, i.e., for $y \notin A$
we utilize the information $Y_i \in A^c$.
The corresponding  dendity is denoted $f^{AC}(y)$ is
$$ f^{AC}(y|\theta)= f^A(y) \;\mbox{if} \; y \in A   \; ; \; f^{AC}(y| \theta)= {P_\theta(A^c)} \; \mbox{if} \; y \in A^c.$$

The right density depends in the context, and on order to simplify notations we will just write $f(y|\theta)$
for all cases $f, f^A, f^{AC}$. The exact version matters for the computation of $\hat{G}$, our estimator
is $\eta_{\hat{G}}=E_{\hat{G}} \eta(\theta)$ for the  appropriate  GMLE $\hat{G}$.

\subsection{Discrete finite support case.}

Consider finite support setups, as  in Section \ref{sec:ex}.  Given $n$, $n=1,2,...$, large enough, every possible outcome is realized.
Hence, by Lemma \ref{lem:main} for every $y$ in the support and for any
pair of sequences of GMLE, $\hat{G}^1 , \hat{G}^2$, \newline
for large enough $n$ :

 \begin{equation} f_{\hat{G}^1}(y)=f_{\hat{G}^2}(y). \label{eq:eqqq} \end{equation} 

 Finally we obtain:
\begin{theorem} \label{thm:main1}
 Under  Assumption \ref{as: cond}, for any  two of sequences GMLE $\hat{G}^1$ and $\hat{G}^2$
$$\lim_{n \rightarrow \infty } [E_{\hat{G}^1} \eta(\theta) -  E_{\hat{G}^2}\eta(\theta)]=0.$$
\end{theorem}

In the above the index $n$ in each sequence is supressed in the notations.

\begin{proof}  The proof follows from  Lemma \ref{lem:main} , and equation \ref{eq:eqqq} 
since 
$$ E_{\hat{G}^i}\eta(\theta)= \int h(y) f_{\hat{G}^i}(y) d\mu.$$

\end{proof}
\bigskip
Note that in the setup of Example \ref{ex:first}  for the case $\eta(\theta)=\theta^2$, the equality  stated in the above theorm does not hold; indeed Assumption \ref{as: cond} does not hold.

\subsection{Infinite support of $Y$.}

In the infinite, and general case of support of $Y$, the idea is similar. 
Consider Reimann/Lebesgue partition of the support into subsets $\Delta_1,...\Delta_M$, $M=M(n) \rightarrow_{n \rightarrow \infty} \infty$. We may choose, e.g., $\Delta_m,\; m=1,...,M(n)$ of equal measure under $\mu$.
For suitable $M(n)= \o(n)$, as $n \rightarrow \infty$, in any $\Delta_m, \; 1,...,M$ there is a realized observation
denoted $\tilde{Y_i}$.
In a similar manner to the case of finite support,  and by the standard theory and conditions of Reimann/Lebesgue  sums and integrability,  
we get:
\begin{theorem} \label {thm:main2}
Under the assumptions of Theorem \ref{thm:main1} for {\it any} $Y$-support size and any two sequnce of GMLE, $\hat{G}^1$ and $\hat{G}^2$:
$$\lim_{n \rightarrow \infty}[ E_{\hat{G}^1} \eta(\theta) -   E_{\hat{G}^2} \eta(\theta) ]=0.$$
\end{theorem}

\section{ Weak convergence of $\hat{\eta}_G$ to $\eta_G$ under non-identifiability.}\label{sec:weak}

In Kiefer and Wolfowitz (1956) and in Chen (2017), conditions are given under which 
$\hat{G} \equiv \hat{G}^n \Rightarrow_{w} G$, where $G$ is the true mixture. Those conditions involve identifiabity assumption.
As mentioned, such a weak convergence would imply $\hat{\eta}_G \rightarrow \eta_G$.

Part i) of the following assumption is (implicitly) analayzed in the above mentioned studies, and  does not require identifiability.  We take it as an assumption and do not provide a proof and conditions for it, especially since it is studied in those papers. Elaborating on that assumption would be a digression from our main point.


 Let $\Gamma^n= \{ \hat{G}^n| \hat{G}^n \mbox{is a GMLE} \} \label{eqn: gamma}$.

\begin{Assumption}\label{as: 2}
i)Assume that for every $n$, there exists $\hat{G}^n_0 \in \Gamma^n$, such that $\hat{G}^n_0 \Rightarrow_w G$,
where $G$ is the true mixing distribution.
\end{Assumption}  

ii) The function $\eta(\theta)$ is continuous and bounded.
\begin{theorem}\label{thm:main3}
Under the assumptions of Theorem \ref{thm:main1} and Assumption \ref{as: 2} it follows that for every
sequence $\hat{G}^n$, $n=1,2,...$, $\hat{G}^n \in \Gamma^n$
$$\hat{\eta}_{G^n}  \rightarrow \eta_G.$$
\end{theorem}

\begin{proof}
 Since $\eta(\theta)$ is assumend bounded and continuous, weak convergence imply $\eta_{\hat{G}^n_0} \rightarrow \eta_G$.
The proof follows by Theorem \ref{thm:main2} , since for any sequence $\hat{G}^n \in \Gamma^n$,  $\eta_{\hat{G}^n_0 } - \eta_{\hat{G}^n} \rightarrow 0$ 
\end{proof}

\section{numerical example}

The example in the following is taken directly from (GR) 2022. Including the remark about our surprize
(at the time)  regarding
the good performance of the estimators $\hat{\eta}_G$.

We study the  Model  described in part a of Example \ref{ex:GR}. The attempted sample size is $\kappa$,
while  $\kappa_i$, the realized sample size from stratum $i$. The variable $\kappa_i$ is
distributed $B(\kappa,\pi_i)$, $\kappa=4$. Our simulated populations  have two types of strata,
500 of each type. In the simulations reported in Table \ref{tab:Binom1},  for all of the 1000 strata, $\eta_G=0.5$ throughout the three simulations summarized in the table.  In 500 strata  $\pi_i=p_i=0.5-\del$ while in the other 500 strata, $\pi_i=p_i=0.5+\del$.

In the tables we  present the mean  and the standard deviation of the  naive estimator  which is based on 
estimation of $p_i$ by $\frac{X_i}{\kappa_i}$ when $\kappa_i>0$, and ignore the cases where 
$\kappa_i=0$, in a missing at random manner. The estimator referred to as GMLE is 
$\hat{\eta}_G$

\begin{table}
\caption{The mean and standard deviation of two estimators. Binomial Simulation. $\kappa = 4$ and $\pi_i=p_i=0.5\pm \del$. }\label{tab2}
\label{tab:Binom1}
\begin{center}
\begin{tabular}{ |c||c|c| }
 \hline
    $\del$ & Naive & GMLE \\
	\hline \hline	
 $0.3$ & {\bf 0.559}, (0.012) & {\bf 0.502}, (0.014) \\
 $0.2$ & {\bf 0.522}, (0.011) & {\bf 0.504}, (0.012) \\
 $0.1$ & {\bf 0.504}, (0.010) & {\bf 0.501}, (0.010) \\
 \hline
\end{tabular}
\end{center}
\end{table}

The final reported simulations are summarized in Table \ref{tab:Binomial2}.
We study Binomial sampling with various values of $\kappa$, fixed at $\kappa=1,\dots,5$. Again, there are
two types of strata, 500 strata for each of
the two types; for 500 strata $p_i$ and $\pi_i$ are sampled independently from  $ U(0.1,0.6)$, while for the rest of the strata they are \iid from $U(0.4,0.9)$.

\begin{table}
\caption{Binomial Simulations with continuous $G$. $\kappa$=1,2,3,4,5.}\label{tab3}
\label{tab:Binomial2}
\begin{center}
\begin{tabular}{ |c||c|c| }
 \hline
    $\kappa$ & Naive & GMLE \\
	\hline \hline	
 $1$ & {\bf 0.544}, (0.019) & {\bf 0.530}, (0.015) \\
 $2$ & {\bf 0.528}, (0.014) & {\bf 0.502}, (0.021) \\
 $3$ & {\bf 0.522}, (0.014) & {\bf 0.498}, (0.022) \\
 $4$ & {\bf 0.517}, (0.012) & {\bf 0.499}, (0.020) \\
 $5$ & {\bf 0.512}, (0.009) & {\bf 0.501}, (0.013) \\
 \hline
\end{tabular}
\end{center}
\end{table}

{\it It is surprising how well the GMLE is doing
already for $\kappa=2,3$, in spite of the
non-identifiability of $G$ and the
inconsistency of the (non-unique) GMLE $\hat{G}$ as an estimator of $G$.

In (GR) the  computation of the GMLE is described. We simply took a dense grid of parameters, and applied EM algorithm. Indeed the algorithm may converge to different GMLE, because of the 
non-identifiability. We simply took in each simulation run the  distribution that the algorithm converged to
in the particular run.

Further simulations  and real data examples may be found in the fore mentioned papers (GI) and (GR).

\subsection{ Weighted average.}

 Given a sample $Y_1,...,Y_n$ with corresponding $\theta_1,...,\theta_n$, a natural estimator for $\eta_G$  would be
$\frac{1}{n}\sum \eta( \theta_i)$, i.e., estimate the expectation under $G$ by the realized sample average. However, $\theta_i$ are not observed.

There could be an interest in estimating
a weighted average. For example when the strata are not of equal size and the population average is not a simple average of the strata's averages.  Similarly when the sampling, of strata or of individuals, is not with
equal probabilities.
Suppose the relative weights are $\gamma_1, ... \gamma_n$, and it desired to estimate $\sum \gamma_i \eta(\theta_i)$, e.g., population's average rather than average of strata's averages.
This section considers  a censored setup, it  just suggests a plausible approach, without claiming
 consistency or any optimality property. 
We only sketch a `plausible' approach.
The sketch will be done in light of the previous sections and their corresponding assumptions.

Given $Y_1,...,Y_n$, assume w.l.o.g that $Y_i, i=1,...,m, \in A$ and $Y_i, \; i=m+1,...,n, \in A^c$.

Now,
\begin{eqnarray} 
\eta_G   &=&  E_{{G}} \eta(\theta) \\
&\approx& \frac{1}{n} [ \; \sum_{i=1}^n \eta(\theta_i)] \\
&\approx& \frac{1}{n}[\sum_{i=1}^m h(Y_i) + \sum_{i=m+1}^n \eta(\theta_i) \;]   .              
\end{eqnarray}

We utilize the law of large numbers.

Denote  $$E=\sum_{i=m+1}^n \frac{1}{n} \eta(\theta_i).     $$
Since $\eta_G \approx \eta_{\hat{G}}, $   we obtain    $\hat{E} \equiv   \eta_{\hat{G}} - \frac{1}{n}\sum_{i=1}^{m} h(Y_i)  \approx E  $
we suggest to estimate $M=\sum_{i=1}^n \gamma_i \eta(\theta_i)$  by 
$$\hat{M}= \sum_{i=1}^{m} \gamma_i h(Y_i) \; + \;  \hat{E  }\times  (\sum _{i=m+1}^n \gamma_i).$$

\newpage

{\bf \Large References:}

\begin{list}{}{\setlength{\itemindent}{-1em}\setlength{\itemsep}{0.5em}}

\item
Chen, J. (2017).  ``Consistency of the MLE under Mixture Models.'' Statist. Sci. 32 (1) 47 - 63.
\item
Efron, B. (2012). Large scale inference: Empirical Bayes  methods for estimation , testing and prediction. Cambridge University press.
\item
Efron, B. (2014). Two modeling strategies for empirical Bayes estimation. Statistical science: a
review journal of the Institute of Mathematical Statistics, 29(2):285.
\item
Greenshtein, E. and Itskov, T (2018), Application of Non-Parametric  Empirical
Bayes to Treatment of Non-Response. {\it Statistica Sinica} 28 (2018), 2189-2208.
\item
Greenshtein, E. and Ritov, Y. (2022). Generalized maximum likelihood estimation of the mean of parameters of mixtures. with applications to sampling and to observational studies. {\it EJS} 16(2): 5934-5954.
\item
Kiefer, J. and Wolfowitz, J. (1956). Consistency of the maximum likelihood estimator in the presence of infinitely many incidental parameters. {\it Ann.Math.Stat.}
27 No. 4, 887-906.
\item
Koenker, R. and Mizera, I. (2014). Convex optimization, shape constraints,
compound decisions and empirical Bayes rules. {\it JASA } 109, 674-685.
\item
Laird, N. (1978). Nonparametric maximum likelihood estimation of a mixing distribution. {\it JASA}  78, No 364, 805-811.
\item
Little, R.J.A and Rubin, D.B. (2002). Statistical Analysis with Missing Data.
New York: Wiley
\item
Robbins, H. (1951). Asymptotically subminimax solutions of compound decision problems. In Proceedings of the Second Berkeley Symposium on Mathematical Statistics and
Probability, 1950 131–148. Univ. California, Berkeley. MR0044803
\item
Robbins, H. (1956). An empirical Bayes approach to statistics. In Proc. Third Berkeley
Symp. 157–164. Univ. California Press, Berkeley. MR0084919
\item
Robbins, H. (1964). The empirical Bayes approach to statistical decision problems. Ann.
Math. Statist. 35 1–20. MR0163407

\item
Zhang, C-H. (2003)  Compound decision theory  and empirical Bayes  methods. Ann. Stat.   31 (2),  379-390.

\end{list}

\end{document}